\pdfoutput=1
\documentclass[reqno]{amsart}
\usepackage{graphicx}
\usepackage{amsthm,amsmath,amssymb}
\usepackage{mathrsfs}
\usepackage{enumerate}
\usepackage{amsaddr}
\newtheorem{thm}{Theorem}[section]

\newtheorem{prop}{Proposition}[section]

\newtheorem{conj}{Conjecture}[section]

\theoremstyle{definition}

\theoremstyle{remark}

\newtheorem{claim}{Claim}
\numberwithin{equation}{section}

\newcommand{\F}{\mathcal{F}}

\begin{document}
\setlength{\parskip}{0in}
\title[]{Large matchings in bipartite graphs have a rainbow matching}%
\author{Daniel Kotlar}%
\address{Computer Science Department, Tel-Hai College, Upper Galilee 12210, Israel}%
\email{dannykot@telhai.ac.il}%
\author{Ran Ziv}%
\address{Computer Science Department, Tel-Hai College, Upper Galilee 12210, Israel}%
\email{ranziv@telhai.ac.il}%
\setlength{\parskip}{0.075in}
\begin{abstract}
Let $g(n)$ be the least number such that every collection of $n$ matchings, each of size at least $g(n)$, in a bipartite graph, has a full rainbow matching. Aharoni and Berger \cite{AhBer} conjectured that $g(n)=n+1$ for every $n>1$. This generalizes famous conjectures of Ryser, Brualdi and Stein. Recently, Aharoni, Charbit and Howard \cite{ACH} proved that $g(n)\le\lfloor\frac{7}{4}n\rfloor$.
We prove that $g(n)\le\lfloor\frac{5}{3} n\rfloor$.
\end{abstract}
\maketitle
\section{Introduction}
Given sets $F_1,F_2,\ldots,F_k$ of edges in a graph a \emph{(partial) rainbow matching} is a choice of disjoint edges from some of the $F_i$'s. In other words, it is a partial choice function whose range is a matching. If the rainbow matching has disjoint edges from all the $F_i$'s then we call it a \emph{full rainbow matching}.
During the last decade the problem of finding conditions for large rainbow matchings in a graph was extensively explored. See, for example, \cite{AhBer, gyarfas2012, kostochka2012, LeSau10}. Many results and conjectures on the subject were influenced by the well-known conjectures of Ryser \cite{Ryser67}, asserting that every Latin square of odd order $n$ has a transversal of order $n$, and Brualdi \cite{Denes74} (see also \cite{BruRys} p. 255), asserting that every Latin square of even order $n$ has a partial transversal of size $n-1$. Brualdi's conjecture may also be casted into the form of a rainbow matching problem:

\begin{conj}\label{conj1}
A partition of the edges of the complete bipartite graph $K_{n,n}$ into $n$ matchings, each of size $n$, has a rainbow matching of size $n-1$.
\end{conj}

A far reaching generalization of Conjecture~\ref{conj1} was posed by Stein \cite{Stein75}:
\begin{conj}\label{conj2}
A partition of the edges of the complete bipartite graph $K_{n,n}$ into n subsets, each of size $n$, has a rainbow matching of size $n-1$.
\end{conj}
Conjecture~\ref{conj2} is equivalent to the assertion that any $n\times n$ array of numbers, where each of the numbers $1,\ldots,n$ appears exactly $n$ times, has a transversal of size $n-1$. The special case where each number appears once in each row translates into a conjecture on matchings, namely, that a family of $n$ matchings, each of size $n$, in a bipartite graph with $2n$ vertices, has a rainbow matching of size $n-1$. We state this conjecture in a more general form, as a generalization of Conjecture~\ref{conj1}:

\begin{conj}\label{conj3}
A family of $n$ matchings, each of size $n$, in a bipartite graph, has a rainbow matching of size $n-1$.
\end{conj}

A slight generalization of results of Woolbright \cite{Wool78} and Brower, de Vries and Wieringa \cite{brouwer78} yields the following theorem (for details see \cite{AKZ}):

\begin{thm}\label{thm1:1}
A family of $n$ matchings, each of size $n$, in a bipartite graph has a rainbow matching of size $n-\sqrt{n}$.
\end{thm}

When the $n$ matchings form a partition of $K_{n,n}$ a tighter bound was achieved by Shor and Hatami \cite{HatShor08}:

\begin{thm}\label{thm1:2}
A partition of edges of the complete bipartite graph $K_{n,n}$ into $n$ matchings, each of size $n$, has a rainbow matching of size $n-O(log^2n)$.
\end{thm}

If one insists on finding a rainbow matching of size $n$ in a family of matchings, each of size $n$, then the size of the family must be raised dramatically, as proved by Drisko \cite{Drisko98}:

\begin{thm}\label{thm1:3}
A family of $2n-1$ matchings, each of size $n$, in a bipartite graph has a rainbow matching of size $n$.
\end{thm}

Drisko provided an example showing that the bound $2n-1$ is tight.

On the other hand, if we want to find a full rainbow matching among $n$ matchings, we may have to increase the size of the matchings. Let $g(n)$ be the least number such that every family of $n$ matchings, each of size at least $g(n)$, in a bipartite graph, has a full rainbow matching. Aharoni and Berger \cite{AhBer} posed the following generalization of Conjecture~\ref{conj3}:

\begin{conj}\label{conj4}
$g(n)=n + 1$.
\end{conj}

It is easy to see, using a greedy algorithm, that $g(n)\le 2n - 1$. Aharoni, Charbit and Howard \cite{ACH} showed:

\begin{thm}\label{thm1:4}
$g(n)\le\lfloor\frac{7}{4}n\rfloor$.
\end{thm}

In Theorem~\ref{thm1} we use a different method to improve the bound in Theorem~\ref{thm1:4} to $\lfloor\frac{5}{3}n\rfloor$.

\section{A rainbow matching of size  $n$}
Let $G$ be a bipartite graph with sides $U$ and $W$.
\begin{prop}\label{prop1}
Let $\F =\{F_1,\ldots,F_n\}$ be a family of $n$ matchings in $G$, where $|F_i|=\lfloor\frac{3}{2}n\rfloor$, $i=1,\ldots,n$. Then, $\F$ has a rainbow matching of size $n-1$.
\end{prop}
\begin{proof}
Assume, by contradiction, that a rainbow matching $R$ of maximal size has size $|R|\le n-2$.
Without loss of generality we may assume that $R\cap F_{n-1}=\emptyset$ and $R\cap F_n=\emptyset$.
Let $X$ and $Y$ be the subsets of $U$ and $W$, respectively, that are not covered by $R$ (see Figure~\ref{fig1}). Since $R$ has maximal size the whole set $X$ is matched by $F_{n-1}$ with some $W'\subset W\setminus Y$ and the whole set $Y$ is matched by $F_n$ with some subset $U'\subset U\setminus X$. Since $|R|\le n-2$ we have $|U'|+|W'|> |R|$. It follows that there exist edges $e_1\in F_{n-1}$, $e_2\in F_n$ and $e\in R$ such that $e_1\cap e\cap W\ne \emptyset$ and $e_1\cap X\ne \emptyset$, and similarly, $e_2\cap e\cap U\ne \emptyset$ and $e_2\cap Y\ne \emptyset$. Clearly, $\left(R\setminus\{e\}\right)\cup\{e_1,e_2\}$ is a rainbow matching, contradicting the maximality property of $R$.
\begin{figure}[h!]
\begin{center}
\includegraphics[scale=0.4]{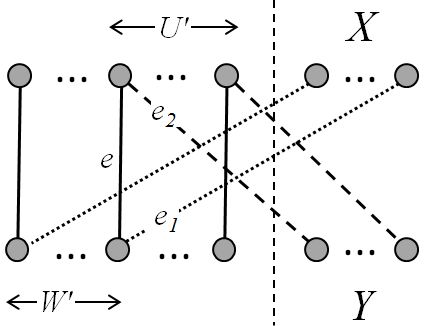}
\end{center}
\caption{Extending the rainbow matching $R$ in Proposition~\ref{prop1}. The solid lines are the edges of $R$, dotted lines represent edges in $F_{n-1}$, and dashed lines represent edges in $F_n$.}
\label{fig1}
\end{figure}

\end{proof}

\begin{thm}\label{thm1}
Let $\F =\{F_1,\ldots,F_n\}$ be a family of $n$ matchings in $G$, each of size $\lfloor\frac{5}{3}n\rfloor$. Then, $\F$ has a full rainbow matching.
\end{thm}

\begin{proof}
Assume, by contradiction, that a rainbow matching $R$ of maximal size satisfies $|R|\le n-1$. From Proposition~\ref{prop1} we know that $|R|=n-1$. Without loss of generality we may assume that $R\cap F_n=\emptyset$. Let $X$ and $Y\subset W$ be the sets of vertices of $G$ not covered by $R$. Since $R$ has maximal size, those vertices in $Y$ that are matched by $F_n$ are matched to $U\setminus X$. Let $Z$ be the set of vertices in $U$ that are matched by $F_n$ to $Y$. Let $R'$ be the subset of $R$ that matches the elements in $Z$ (see Figure~\ref{fig2}). We have
\begin{equation}\label{eq1:1}
    |R'|= \lfloor2n/3\rfloor+1
\end{equation}
Define,
\begin{equation*}
    \F'=\{F_i\in\F | F_i\cap R'\ne \emptyset\}.
\end{equation*}

\begin{figure}[h!]
\begin{center}
\includegraphics[scale=0.4]{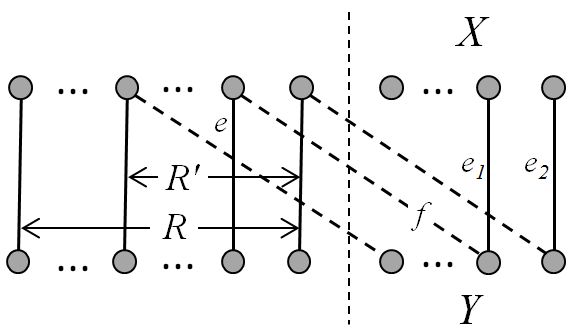}
\end{center}
\caption{The sets $R$ and $R'$. Dashed lines represent the edges of $F_n$ matching vertices of $Y$.}
\label{fig2}
\end{figure}

\begin{claim}\label{claim:1}
Any matching $F_i\in\F'$ has at most one edge between $X$ and $Y$.
\end{claim}

\begin{proof}[Proof of Claim~\ref{claim:1}]
\renewcommand{\qedsymbol}{}
Suppose $F_i\cap R'= \{e\}$ and $F_i$ has two edges $e_1$ and $e_2$ between $X$ and $Y$. Let $f$ be the edge of $F_n$ such that $f\cap e\ne \emptyset$ and $f\cap Y\ne \emptyset$. Without loss of generality we may assume that $f$ does not meet $e_2$ (it may or may not meet $e_1$). Thus, $\left(R\setminus \{e\}\right)\cup\{f,e_2\}$ is a rainbow matching of size $n$ (Figure~\ref{fig2}), contradicting the maximality of $R$.
\end{proof}
As a consequence of Claim~\ref{claim:1} we have:

\begin{claim}\label{claim:2}
Each $F_i\in\F'$ has at least $\lfloor2n/3\rfloor$ edges with one endpoint in $X$ and the other endpoint in $W\setminus Y$ and at least $\lfloor2n/3\rfloor$ edges with one endpoint in $Y$ and the other endpoint in $U\setminus X$.
\end{claim}

Now, let $W'$ be the set of those vertices in $W$ that are endpoints of edges in $R'$.

\begin{claim}\label{claim:3}
Each $F_i\in\F'$ has at least $\lceil n/3\rceil$ edges with one endpoint in $X$ and the other endpoint in $W'$.
\end{claim}

\begin{proof}[Proof of Claim~\ref{claim:3}]
\renewcommand{\qedsymbol}{}
By (\ref{eq1:1}), $|R\setminus R'|= n-1- (\lfloor2n/3\rfloor+1)=\lceil n/3\rceil-2$. By Claim~\ref{claim:2}, the edges with one endpoint in $X$ and the other endpoint in $W\setminus Y$ meet at least $\lfloor2n/3\rfloor - (\lceil n/3\rceil-2)\ge\lceil n/3\rceil$ edges of $R'$.
\end{proof}

Without loss of generality we assume that $F_1\in\F'$. Let $F_1\cap R'=\{r_1\}$ and
let $e_1\in F_1$ be such that $e_1\cap X\ne \emptyset$ and $e_1\cap W'\ne \emptyset$ (such $e_1$ exists by Claim~\ref{claim:3}).
Let $r_2\in R'\setminus\{r_1\}$ be an edge such that $r_2\cap e_1\ne \emptyset$ (See Figure~\ref{fig2a}). We may assume $r_2\in F_2$. By Claim~\ref{claim:3}, $F_2$ has at least $\lceil n/3\rceil-1$ edges with one endpoint in $X\setminus\{e_1\}$ and the other endpoint in $W'$.
Let $e_2\in F_2$ be such an edge, that is, $e_2\cap X\ne \emptyset$, $e_2\cap W'\ne \emptyset$ and $e_2\cap e_1= \emptyset$. Clearly, $e_2\cap r_2= \emptyset$, since they belong to the same matching. If $e_2\cap r_1\ne \emptyset$, then we can augment $R$ in the following way: Let $f\in F_n$ be the edge satisfying $f\cap r_1\cap U\ne \emptyset$ and $f\cap Y\ne \emptyset$ ($f$ exists since $F_1\in\F'$). Then $\left(R\setminus\{r_1,r_2\}\right)\cup\{e_1,e_2,f\}$ is a rainbow matching of size $n$ (see Figure~\ref{fig2a}). If $e_2\cap r_1= \emptyset$, let $r_3\in R'\setminus\{r_1,r_2\}$ be such that $e_2\cap r_3\ne \emptyset$. We may assume $r_3\in F_3$.

\begin{figure}[h!]
\begin{center}
\includegraphics[scale=0.4]{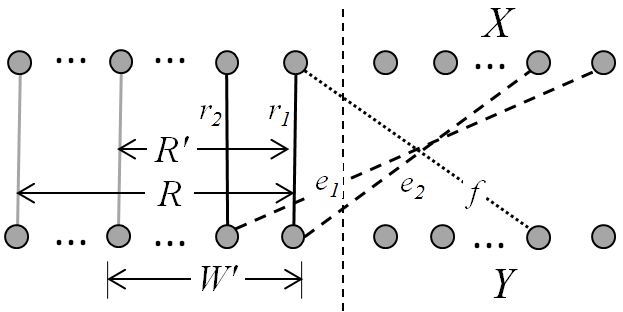}
\end{center}
\caption{Augmenting the rainbow matching $R$. Solid edges are omitted, dashed and dotted edges are added.}
\label{fig2a}
\end{figure}

 We proceed in this manner to obtain a set of disjoint edges $E=\{e_1,e_2,\ldots,e_k\}$, each with endpoints in $X$ and in $W'$ and a set of edges $R''=\{r_1,r_2,\ldots,r_k, r_{k+1}\}\subseteq R'$ such that $e_i\cap r_{i+1}\ne \emptyset$ for $i=1,\ldots,k$ (Figure~\ref{fig3}), and for each $i$, $e_i$ and $r_i$ belong to the same matching. Without loss of generality we assume that $e_i,r_i\in F_i$, for $i=1,\ldots,k$, and $r_{k+1}\in F_{k+1}$. If  $k<\lceil n/3\rceil$, then the matching $F_{k+1}$ still has edges with one endpoint in $X\setminus\{e_1,e_2,\ldots,e_k\}$ and the other endpoint in $W'$, by Claim~\ref{claim:3}. Suppose one of these edges, say $e_{k+1}\in F_{k+1}$, also satisfies $e_{k+1}\cap r_t\ne\emptyset$ for some $t\in\{1,\ldots,k\}$. Then $R$ can be augmented as follows: take $f\in F_n$ such that $f\cap Y\ne \emptyset$ and $f\cap r_l\ne \emptyset$ for some $l\in\{t,\ldots,k+1\}$. Then $\left(R\setminus\{r_t,\ldots,r_{k+1}\}\right)\cup \{e_t,\dots,e_{k+1},f\}$ is a rainbow matching of size $n$ (see Figure~\ref{fig4}). Otherwise, some $e_{k+1}\in F_{k+1}$ has one endpoint in $X\setminus\{e_1,e_2,\ldots,e_k\}$ and one endpoint in $W'\setminus \{r_1,r_2,\ldots,r_k, r_{k+1}\}$, so the path can be extended. Thus, we may assume that $k\ge\lceil n/3\rceil$.

 \begin{figure}[h!]
\begin{center}
\includegraphics[scale=0.4]{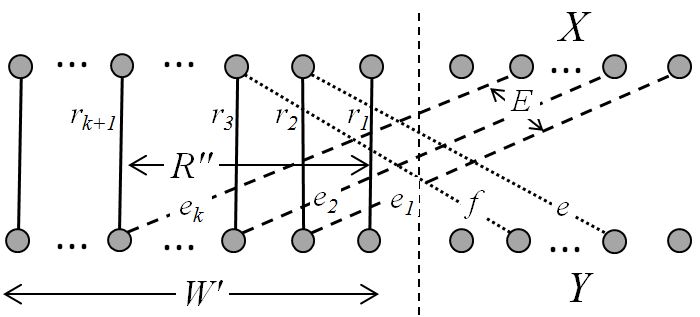}
\end{center}
\caption{Construction of the sets $E$ and $R''$ and augmenting the rainbow matching $R$ in Case 2. Solid dark edges are omitted, dashed and dotted edges are added.}
\label{fig3}
\end{figure}

\begin{figure}[h!]
\begin{center}
\includegraphics[scale=0.4]{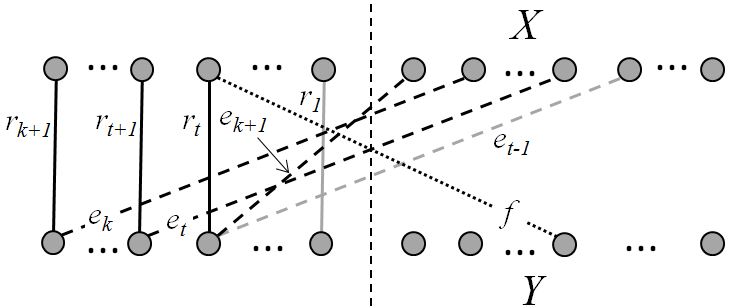}
\end{center}
\caption{Augmenting the rainbow matching $R$ in Case 1. Solid dark edges are omitted. Dashed and dotted dark edges are added.}
\label{fig4}
\end{figure}

 Now, the set of edges $\left(R\setminus R''\right)\cup E$ forms a partial rainbow matching of size $n-2$. It excludes the matchings $F_{k+1}$ and $F_n$. Since we assume that $k\ge\lceil n/3\rceil$, we have
 \begin{equation}\label{eq2:2}
    |R''|>\lceil n/3\rceil
 \end{equation}
 and thus, $|R\setminus R''|< n-1- \lceil n/3\rceil=\lfloor 2n/3\rfloor-1<\lfloor 2n/3\rfloor$. By Claim~\ref{claim:1}, there are at least $\lfloor2n/3\rfloor$ edges of $F_{k+1}$ with endpoints in $Y$ and $U\setminus X$. Thus, there exists an edge $e\in F_{k+1}$ such that $e\cap Y\ne\emptyset$ and $e\cap r_i\ne\emptyset$ for some $r_i\in R''$. Let $f\in F_n$ be such that $f\cap Y\ne\emptyset$, $e\cap f=\emptyset$, and $f\cap r_j\ne\emptyset$ for some $r_j\in R''\setminus\{r_i\}$. Since Theorem~\ref{thm1} can be easily verified for $n\le3$, we assume that $n>3$. Then, by (\ref{eq2:2}), $|R''|>2$. Since all the vertices in $R''$ are matched by $F_n$ to vertices in $Y$, such $f$ exists. Hence, $\left(R\setminus R''\right)\cup E\cup\{e,f\}$ is a rainbow matching of size $n$ (Figure~\ref{fig3}). This completes the proof.

\end{proof}

\providecommand{\href}[2]{#2}

\end{document}